\documentclass[a4paper,reqno]{amsart}
\usepackage{amsmath}
\usepackage{amsthm}
\usepackage{amsfonts}
\usepackage{amssymb}
\usepackage{color}
\usepackage{mathrsfs}
\usepackage{graphicx}
\usepackage{booktabs}
\usepackage{longtable}

\numberwithin{equation}{section}

\theoremstyle{plain}
\newtheorem{theorem}{Theorem}

\theoremstyle{plain}
\newtheorem*{theorem0}{Theorem A}

\theoremstyle{plain}
\newtheorem*{conjecture0}{Conjecture A}

\theoremstyle{plain}
\newtheorem{lemma}{Lemma}

\theoremstyle{plain}
\newtheorem{claim}{Claim}

\theoremstyle{plain}

\theoremstyle{plain}
\newtheorem{corollary}{Corollary}

\theoremstyle{plain}
\newtheorem*{definition}{Definition}

\theoremstyle{definition}

\theoremstyle{definition}

\theoremstyle{definition}
\newtheorem*{case1}{Case 1}

\theoremstyle{definition}
\newtheorem*{case2}{Case 2}

\theoremstyle{definition}
\newtheorem*{case3}{Case 3}

\theoremstyle{definition}
\newtheorem*{case31}{Case 3.1}

\theoremstyle{definition}
\newtheorem*{case32}{Case 3.2}

\theoremstyle{definition}
\newtheorem*{case33}{Case 3.3}

\theoremstyle{definition}
\newtheorem*{case34}{Case 3.4}

\theoremstyle{definition}
\newtheorem*{case35}{Case 3.5}

\theoremstyle{definition}
\newtheorem*{case351}{Case 3.5.1}

\theoremstyle{definition}
\newtheorem*{case352}{Case 3.5.2}

\DeclareMathOperator{\ch}{ch}

\renewcommand{\le}{\leqslant}
\renewcommand{\ge}{\geqslant}

\title{Optimal Locating-Paired-Dominating Sets in King Grids}

\author{Yuxuan Yang}

\address{Department of Science, Beijing University of Posts and Telecommunications, Beijing, China}

\email{yangyx@bupt.edu.cn}



\begin{document}

\begin{abstract}
In this paper, we continue the study of locating-paired-dominating set, abbreviated LPDS, in graphs introduced by McCoy and Henning. Given a finite or infinite graph $G=(V,E)$, a set $S\subset V$ is paired-dominating if the induced subgraph $G[S]$ has a perfect matching and every vertex in $V$ is adjacent to a vertex in $S$. The other condition for LPDS requires that for any distinct vertices $u,v \in V\backslash S$, we have $N(u)\cap S\neq N(v)\cap S$. Motivated by the conjecture of Kinawi, Hussain and Niepel, we prove the minimal density of LPDS in the king grid is between $8/37$ and $2/9$, and we find uncountable many different LPDS with density $2/9$ in the king grid. These results partially solve their conjecture.
\end{abstract}

\maketitle

\thispagestyle{empty}

%%%%%%%%%%
%
% SECTION 1
%
%%%%%%%%%%

\section{Introduction}\label{sec1}
The concept of locating-paired-dominating sets was introduced by McCoy and Henning \cite{lpds} as an extension of paired-dominating sets \cite{pds1,pds2}. The problem of locating monitoring devices in a system when every monitor is paired with a backup monitor serves as the motivation for this concept.

Here are some basic definitions and notations. Let $G=(V, E)$ be an undirected graph. For a subset $S\subset V$, $G[S]$ is a subgraph of $G$ induced by $S$. For a vertex $v$ in $G$, the set $N_G(v)=\{u\in V|uv\in E\}$ is called the \textit{open neighborhood} of $v$ and $N_G[v]=N_G(v) \cup \{v\}$ is the \textit{closed neighborhood} of $v$. We may use $N(v)$ and $N[v]$ for short if there is no ambiguity of the base graph $G$. And a set $S$ of vertices is called \textit{dominating set} of $G$ if for any $v\in V$, $N[v]\cap S \neq \emptyset$. A dominating set is called \textit{locating-dominating set} if for any pair of distinct vertices $v,u\in V\setminus S$, $N[u]\cap S\neq N[v]\cap S$. A set $M$ of edges is called \textit{matching} if they do not share vertices for each two. A matching $M$ is a \textit{perfect matching} if every vertex is incident to an edge of $M$. A \textit{paired-dominating set} of $G$ is a dominating set $S$ whose induced subgraph $G[S]$ contains a perfect matching (not necessarily induced). The set $S$ is said to be a \textit{locating-paired-dominating set}, abbreviated LPDS, if $S$ is both a paired-dominating set and a locating-dominating set. 

It is optimal to make the size of the dominating set as small as possible. The problem of finding an optimal locating-dominating set in an arbitrary graph is known to be NP-complete \cite{nphard}. Particular interest was dedicated to infinite grids as many processor networks have a grid topology \cite{pds6} \cite{pds3} \cite{pds4} \cite{pds5}. There was also the corresponding study for LPDS. For instance, LPDS in infinite square grids was discussed in \cite{lpdssqr}, and LPDS in infinite triangular and king grids was discussed in \cite{lpdsking}. Especially, Hussain, Niepel, and Kinawi guessed the density of an optimal LPDS in the king grid is equal to $2/9$ in \cite{lpdsking}, and they proved this number is between $3/14$ and $2/9$. Our results partially solve this conjecture. For example, we use the Discharge Method to improve the lower bound to $8/37$. See Section 3 for other main results.

%%%%%%%%%%%%%%%
%
% SECTION 2
%
%%%%%%%%%%%%%%%
\section{preliminaries}
In this section, we list several notations and backgrounds for this paper. Some of them come from \cite{lpdsking}.
\begin{definition}[king grid]
An infinite graph $G=(V,E)$ is called king grid, where $V=Z\times Z$. Two vertices are adjacent if and only if their Euclidean distance is less or equal to $\sqrt{2}$.
\end{definition}
Each vertex $v$ in the king grid has $8$ neighbors, and four of its neighbors are with Euclidean distance $\sqrt{2}$ from $v$. We call them \textit{$\sqrt{2}$-neighbors} of $v$. If two $\sqrt{2}$-neighbors of $v$ have distance $2\sqrt{2}$, we call them in the \textit{opposite position} of each other. The \textit{distance} of two vertices $u,v$ of a graph is the number of the edges in the shortest path connecting $u$ and $v$, which is denoted by $d(u,v)$.
\begin{definition}[k-neighborhood]
The k-neighborhood of vertex $u$ in graph $G$ is defined as $N_G^k[u]=\{x\in V(G)|d(u,x)\leq k\}$.
\end{definition}
To avoid unnecessary complexity, in this paper when we consider infinite graphs, we always assume there is a universal upper bound $M\in \mathbb{Z}_+$ of the degree, which means for all $v\in V$, $|N_G(v)|\le M$. The king grid satisfies this condition for sure.
\begin{definition}[density]
Given a graph $G=(V,E)$, which can be finite or infinite, and a set $S\subset V$, the density of $S$ in $G$ is defined as 
\begin{equation}
\label{term1}
    D(S)=\limsup_{k\rightarrow \infty} \frac{|S\cap N_G^k[u]|}{|N_G^k[u]|},
\end{equation}
where $u$ is a vertex of G.
\end{definition}
Note that the density $D(S)$ is independent of the choice of the vertex $u$. 
%Similarly, if we have a function 
%\begin{equation*}
%    \ch: V\rightarrow \Rr
%\end{equation*}
%on the vertex set of an infinite graph, we can define the average value of this function as
%\begin{equation*}
%    \limsup_{k\rightarrow \infty} \frac{\sum_{v\in N_G^k[u]}\ch(v)}{|N_G^k[u]|}.
%\end{equation*}

Hussain, Niepel, and Kinawi have proved the following result for the density of an optimal LPDS in the king grid \cite{lpdsking}. Here, the term ``optimal" means the LPDS has minimal density. 

\begin{theorem0}
\label{thm0}
If $S$ is an optimal LPDS in the king grid and $D(S)$ is the density of $S$, then we have $3/14\le D(S)\le 2/9$.
\end{theorem0}

Also, they gave a conjecture for this density \cite{lpdsking}.

\begin{conjecture0}
\label{conj}
If $S$ is an optimal LPDS in the king grid and $D(S)$ is the density of $S$, then we have $D(S)= 2/9$.
\end{conjecture0}
In this paper, we always assume $G=(V,E)$ is the king grid and $S$ is a LPDS of $G$. For convenience, we use the following notations. Fix an arbitrary perfect matching $M$ on the induced subgraph $G[S]$. Denote
\begin{equation}
\label{term2}
    m:S\rightarrow S
\end{equation}
to be a map that sends each vertex to its partner in $M$. In other words, we have $N_M(v)=\{m(v)\}$. Because of the lack of edge-transitivity in the king grids, we have two different types of pairs in $M$. We define
\begin{equation}
\label{term3}
    \begin{aligned}
    S_1:=\{v\in S|\lvert N(v)\cap N(m(v))\rvert=2\},\\
    S_2:=\{v\in S|\lvert N(v)\cap N(m(v))\rvert=4\}.
    \end{aligned}
\end{equation}
Remark that each pair in $S_1$ has Euclidean distance $\sqrt{2}$ and each pair in $S_2$ has Euclidean distance $1$, so we call them far pairs and close pairs, respectively. It's trivial to find $S=S_1\cup S_2$ and $S_1\cap S_2=\emptyset$.

%%%%%%%%%%%%%%%
%
% SECTION 3
%
%%%%%%%%%%%%%%%
\section{main results}
Suppose $S$ is a LPDS of the king grid $G$, under the notation \eqref{term3} in the last section we have the following observation for the density of $S_1$ and $S_2$. 
\begin{theorem}
\label{thm1}
Suppose $S_1$ and $S_2$ have density $D(S_1)$ and $D(S_2)$ in $G$, then
\begin{equation}
\label{eqthm1}
    \frac{14}{3} D(S_1)+\frac{9}{2} D(S_2)\ge 1.
\end{equation}
\end{theorem}

\begin{theorem}
\label{thm2}
Suppose $S_1$ and $S_2$ have density $D(S_1)$ and $D(S_2)$ in $G$, then
\begin{equation}
\label{eqthm2}
    \frac{9}{2} D(S_1)+5 D(S_2)\ge 1.
\end{equation}
\end{theorem}

Remark that Theorem \ref{thm1} has almost been indicated in \cite{lpdsking}, in which the idea called ``share" was used. We prove both Theorem \ref{thm1} and Theorem \ref{thm2} by using \textit{Discharge Method}. The latter one is the main work of this paper.

From these two theorems, we can easily improve the lower bound in Theorem A.
\begin{corollary}
The density of an optimal LPDS in the king grid is at least 8/37.
\end{corollary}
\begin{proof}
Suppose $S$ is a LPDS in $G$, then we have
\begin{equation*}
    \begin{aligned}
     D(S)&=D(S_1)+D(S_2)\\
     &=\frac{3[\frac{14}{3}D(S_1)+\frac{9}{2}D(S_2)]+[\frac{9}{2}D(S_1)+5D(S_2)]}{37/2}\\
     &\ge\frac{3+1}{37/2}\\
     &=\frac{8}{37}.
    \end{aligned}
\end{equation*}
\end{proof}

Moreover, if we have a counterexample for Conjecture A, the fair pairs and close pairs must both have positive density.
\begin{corollary}
  Suppose $S$ is a LPDS in $G$ and $D(S)<2/9$, then $D(S_1)>0$ and $D(S_2)>0$.
\end{corollary}
\begin{proof}
\begin{equation*}
\begin{aligned}
    1&\le \frac{14}{3}D(S_1)+\frac{9}{2}D(S_2)\\
    &=\frac{14}{3}D(S_1)+\frac{9}{2}(D(S)-D(S_1))\\
    &=\frac{9}{2}D(S)+\frac{1}{6}D(S_1)\\
    &< \frac{9}{2}\cdot \frac{2}{9}+\frac{1}{6}D(S_1).\\
    1&\le \frac{9}{2}D(S_1)+5D(S_2)\\
    &=\frac{9}{2}(D(S)-D(S_2))+5D(S_2)\\
    &=\frac{9}{2}D(S)+\frac{1}{2}D(S_2)\\
    &< \frac{9}{2}\cdot \frac{2}{9}+\frac{1}{2}D(S_2).
\end{aligned}
\end{equation*}
\end{proof}

In other words, a LPDS in the king grid has density at least $9/2$, if the pairs in this LPDS are either all far pairs or all close pairs.

All the results above are for the lower bound of the density of an optimal LPDS in the king grid. For the upper bound, after a lot of attempts, we believe that Conjecture A is true, which means the upper bound cannot be improved. However, we find the LPDS with density 2/9 given by Hussain, Niepel, and Kinawi \cite{lpdsking} is not the only pattern.

\begin{theorem}
\label{thm4}
There are uncountable many different LPDS with density 2/9 in the king grid.
\end{theorem}
We leave the construction of those patterns in Section 5.
%%%%%%%%%%%%%%%
%
% SECTION 4
%
%%%%%%%%%%%%%%%
\section{Discharge Methods}
In this section, we prove Theorems \ref{thm1} and \ref{thm2} using Discharge Method. 
Suppose $S$ is a LPDS of the king grid $G=(V,E)$, we follow the notation \eqref{term3} about $S_1, S_2$. For each $v\in S$, we define its pendent neighbors and interval neighbors as
\begin{equation}
\label{term4}
    \begin{aligned}
     &P(v):=N(v)\backslash N[m(v)],\\
     &I(v):=N(v)\cap N(m(v)).
    \end{aligned}
\end{equation}
It is trivial to see that when $v\in S_1$, we have $\lvert P(v)\rvert =5 $ and $\lvert I(v) \rvert =2$. Also, when $v\in S_2$, we have $\lvert P(v)\rvert =3 $ and $\lvert I(v) \rvert =4$.

For each $v\in V\backslash S$, we assign it to one of the following sets according to the size of the set $N(v) \cap S$.
\begin{equation}
\label{term5}
    \begin{aligned}
     T_1:=\{v\in V\backslash S | \lvert N(v) \cap S\rvert =1\},\\
     T_2:=\{v\in V\backslash S | \lvert N(v) \cap S\rvert =2\},\\
     T_3:=\{v\in V\backslash S | \lvert N(v) \cap S\rvert \ge 3\}.
    \end{aligned}
\end{equation}
We have $V=T_1\cup T_2\cup T_3\cup S$, since $S$ is dominating.
At first, we investigate the neighbors for each dominating vertex $v\in S$.
\begin{lemma}
\label{lem1}
We have the following observations.

(1) $\forall v\in S, \lvert P(v)\cap T_1\rvert \le 1$,

(2) $\forall v\in S,\lvert I(v) \cap T_2\rvert \le 1, I(v) \cap T_1=\emptyset$,

(3) $\forall v\in S_1, \lvert P(v) \cap (T_3\cup S)\rvert \ge 1$.
\end{lemma}
\begin{proof}
(1) Suppose $v$ is the only neighbor of two vertices in $P(v)\backslash S$, then it contradicts the locating property for LPDS.

(2) Suppose $v$ and $m(v)$ are the only neighbors of two vertices in $I(v)\backslash S$, then it contradicts the locating property for LPDS. Also the vertex in $I(v)$ has at least two neighbors, $v$ and $m(v)$, in $S$.

(3) We assume $P(v)\cap (T_3\cup S)=\emptyset$ and $v\in S_1$. From the symmetry, it is safe to assume $v=(0,0)$ and $m(v)=(-1,-1)$. From (1), we know either $(1,0)$ or $(0,1)$ is not in $T_1$. From the symmetry, it is safe to assume $u=(0,1)$ is not in $T_1$. Then $u$ has another neighbor in $S$, say $v^{\prime}$. It is not hard to find that $N(v)\cap N(v^{\prime})$ has at least two members. These two members have to be in $T_2$, which contradicts the locating property for LPDS.
\end{proof}
Remark that Lemma \ref{lem1} have been covered by \cite{lpdsking}.
\begin{proof}[Proof of Theorem \ref{thm1}]
The initial charge $\ch_0$ on $G$ is given by 
\begin{align*}
    &\forall v\in V\backslash S, \ch_0(v)=0,\\
    &\forall v\in S_1, \ch_0(v)=14/3,\\
    &\forall v\in S_2, \ch_0(v)=9/2.
\end{align*}
Then we apply the following local discharge rules.
\begin{align*}
   \forall i\in\{1,2,3\},\forall v\in S, \forall u\in T_i\cap N(v), f(v,u)=1/i,
\end{align*}
where $f$ is a function from $V\times V$ to $[ 0,+\infty)$ and the function value $f(v,u)$ means the charge that sends from $v$ to $u$. Remark that $f(v,u)=0$ if it is not mentioned above.

The new charge is given by 
\begin{equation*}
    \ch_1(v)=\ch_0(v)+\sum_{u\in N(v)}f(u,v)-\sum_{u\in N(v)}f(v,u),\forall v\in V.
\end{equation*}

We claim that the new charge is at least one everywhere in the king grid.
\begin{claim}
\label{cl1}
Under the discharge rules $f$, for each $v\in V$, we have $\ch_1(v)\ge 1$.
\end{claim}
\begin{proof}[Proof of Claim \ref{cl1}]
We have $V=T_1\cup T_2\cup T_3\cup S_1 \cup S_2$. For each $i\in\{1,2,3\}$, if $v\in T_i$, then $v$ get $1/i$ for each of its at least $i$ neighbors in $S$, which makes 
\begin{equation*}
    \ch_1(v)\ge 0+i\cdot 1/i=1.
\end{equation*}
If $v\in S_1$, then from Lemma \ref{lem1}(1),(2),(3), we have
\begin{align*}
    \ch_1(v)&=\frac{14}{3}-\sum_{u\in P(v)}f(v,u)-\sum_{u\in I(v)} f(v,u)\\
    &\ge \frac{14}{3}-(1+\frac{1}{3}+3\cdot\frac{1}{2})-(\frac{1}{2}+\frac{1}{3})\\
    &=1.
\end{align*}
If $v\in S_2$, then from Lemma \ref{lem1}(1),(2), we have
\begin{align*}
    \ch_1(v)&=\frac{9}{2}-\sum_{u\in P(v)}f(v,u)-\sum_{u\in I(v)} f(v,u)\\
    &\ge \frac{9}{2}-(1+2\cdot\frac{1}{2})-(\frac{1}{2}+3\cdot\frac{1}{3})\\
    &=1.
\end{align*}
\end{proof}
Under the definition of density on an infinite graph, it is not hard to find that the average charge for $\ch_0$ is the left side of \eqref{eqthm1} and the average charge for $\ch_1$ is at least 1. The average charge is fixed under the discharge rule, so Theorem \ref{thm1} holds.
\end{proof}
\begin{proof}[Proof of Theorem \ref{thm2}]
The new initial charge $\ch_2$ on $G$ is given by
\begin{align*}
    &\forall v\in V\backslash S, \ch_2(v)=0,\\
    &\forall v\in S_1, \ch_2(v)=9/2,\\
    &\forall v\in S_2, \ch_2(v)=5.
\end{align*}
Similarly, it suffices to give several local discharge rules to make the charge at least 1 everywhere. We will manage it in three steps.
\begin{equation}
    \ch_2\xrightarrow{g_1}\ch_3\xrightarrow{g_2}\ch_4\xrightarrow{g_3}\ch_5.
\end{equation}

For convenience, we would like to add more notations before the discharge. We collect all the interval neighbors as $I$ and define the redundant pendent neighbors of $v$ denoted by $P_0(v)$.
\begin{equation*}
    I:=\cup_{v\in S} I(v),\qquad P_0(v):=P(v)\cap (S\cup I).
\end{equation*}
The remaining pendent neighbors of $v$ belong to one of the following sets.
\begin{equation*}
    \begin{aligned}
        P_i(v):=(P(v)\backslash P_0(v))\cap T_i,\quad i\in\{1,2,3\}.
    \end{aligned}
\end{equation*}
The size of these sets is denoted by 
\begin{equation*}
    p_i(v):=|P_i(v)| \quad \forall i\in\{0,1,2,3\}.
\end{equation*}
The members in $I(v)$ which is also in $S$ form the set $I_0(v)$ and its size is denoted by $i_0(v)$.
\begin{equation*}
    I_0(v):=I(v)\cap S,\quad i_0(v)=|I_0(v)|.
\end{equation*}
Here comes the first-step discharge rule $g_1(v,u)$.
\begin{align*}
    &\forall v\in S, \forall u\in I(v)\backslash S, g_1(v,u)=1/2,\\
    &\forall i\in\{1,2\},\forall v\in S, \forall u\in P_i(v), g_1(v,u)=1/i.
\end{align*}
Under the discharge $g_1$, we make sure that the vertices in $S\cup I\cup T_1\cup T_2$ have enough charge.
\begin{claim}
\label{cl2}
Under the discharge rules $g_1$, for each $v\in S\cup I\cup T_1\cup T_2$,  $\ch_3(v)\ge 1$.
\end{claim}
\begin{proof}[Proof of Claim \ref{cl2}]
For each $u\in I\backslash S$, there exists at least two different $v_1,v_2\in S$ such that $u\in I(v_1)$ and $u\in I(v_2)$, then
\begin{equation*}
    ch_3(u)\ge 0+2\cdot1/2=1.
\end{equation*} 
For each $u\in T_1\backslash I$, there exists a unique $v\in S$ such that $u\in  P_1(v)$, then $\ch_3(u)\ge 1$. For each $u\in T_2\backslash I$, there exist two different $v_1,v_2\in S$ such that $u\in P_2(v_1)$ and $u\in P_2(v_2)$, then $\ch_3(u)\ge 1$. For each $v\in S_1$, from Lemma \ref{lem1}(1),(2),(3), we have 
\begin{equation*}
\begin{aligned}
    \ch_3(v)&=\frac{9}{2}-\sum_{u\in P(v)}g_1(v,u)-\sum_{u\in I(v)} g_1(v,u)\\
    &\ge \frac{9}{2}-(1+3\cdot\frac{1}{2})-(2\cdot\frac{1}{2})\\
    &=1.
\end{aligned}
\end{equation*}
For each $v\in S_2$, from Lemma \ref{lem1}(1),(2), we have 
\begin{equation*}
\begin{aligned}
    \ch_3(v)&=5-\sum_{u\in P(v)}g_1(v,u)-\sum_{u\in I(v)} g_1(v,u)\\
    &\ge 5-(1+2\cdot\frac{1}{2})-(4\cdot\frac{1}{2})\\
    &=1.
\end{aligned}
\end{equation*}
\end{proof}
The remaining task for discharge $g_2$ and $g_3$ is to sends proper charge to $T_3\backslash I=\cup_{v\in S} P_3(v)$. We need to make sure the charge is at least one everywhere on the infinite kind grid.

A natural idea is to send the redundant charge of $S$ to its neighbors $P_3(v)$ uniformly, but we would like to add an upper bound for it manually. For each $v\in S$, let
\begin{equation*}
    r(v):=\min \{\frac{\ch_3(v)-1}{p_3(v)},\frac{1}{2}\}.
\end{equation*}
Note that $r(v)\ge 0$ since $\ch_3(v)\ge 1$ when $v\in S$. Next, the second-step discharge rule $g_2(v,u)$ is given by
\begin{equation}
    \forall v\in S ,\forall u\in P_3(v), g_2(v,u)=r(v).
\end{equation}
From the discharge rule, its trivial to see that $\ch_4(v)\ge 1$ for each $v\in S$. Also, we have the following observations.
\begin{claim}
\label{lemr1}
If one of the following conditions holds, then $r(v)=1/2$.

(1) $v\in S_2$, (2) $p_0(v)+i_0(v)\ge 1$, (3) $p_1(v)=0$.

\end{claim}
\begin{proof}[Proof of Claim \ref{lemr1}]
If $v\in S_2$, from Lemma \ref{lem1}(1) and $p_1(v)+p_2(v)+p_3(v)\le |P(v)|=3$, we have
\begin{align*}
\ch_3(v)&\ge 5-\frac{1}{2}|I(v)\backslash S|-1\cdot p_1(v)-\frac{1}{2}p_2(v)\\
&\ge 5-\frac{1}{2}\cdot 4-\frac{1}{2}-\frac{1}{2}(3-p_3(v))\\
&=1+\frac{1}{2}p_3(v),
\end{align*}
which implies $r(v)= 1/2$.

If $v\in S_1$ and $p_0(v)+i_0(v)\ge 1$, from Lemma \ref{lem1}(1) and 
\begin{equation*}
p_1(v)+p_2(v)+p_3(v)\le |P(v)|-p_0(v)\le 5-p_0(v),    
\end{equation*}
we have
\begin{align*}
    \ch_3(v)&\ge \frac{9}{2}-\frac{1}{2}|I(v)\backslash S|- 1\cdot p_1(v)-\frac{1}{2} p_2(v)\\
    &\ge \frac{9}{2}-\frac{1}{2}(2-i_0(v))-\frac{1}{2}-\frac{1}{2}(5-p_0(v)-p_3(v))\\
    &=1+\frac{1}{2}p_3(v),
\end{align*}
which implies $r(v)= 1/2$.

Similarly, if $v\in S_1$ and $p_1(v)=0$, then
\begin{align*}
    \ch_3(v)&\ge \frac{9}{2}-\frac{1}{2}|I(v)\backslash S|- \frac{1}{2}p_2(v)\\
    &\ge \frac{9}{2}-\frac{1}{2}\cdot 2-\frac{1}{2}(5-p_3(v))\\
    &=1+\frac{1}{2}p_3(v),
\end{align*}
which implies $r(v)= 1/2$.
\end{proof}
\begin{claim}
\label{lemr2}
For each $v \in S$, we have 
\begin{align*}
    r(v)\ge \frac{p_3(v)-1}{2p_3(v)}
\end{align*}
\end{claim}
\begin{proof}[Proof of Claim \ref{lemr2}]
According to Claim \ref{lemr1}, in those cases $r(v)\ge \frac{1}{2}\ge \frac{p_3(v)-1}{2p_3(v)}$. It suffices to check the case that satisfies $v\in S_1$, $p_0(v)=0$ and $p_1(v)=1$. From these three conditions and Lemma \ref{lem1}(3), we know that $5=p_1(v)+p_2(v)+p_3(v)$, $p_3(v)\ge 1$. Then, we have
\begin{align*}
    \ch_3(v)&=\frac{9}{2}-\frac{1}{2}|I(v)\backslash S|-1\cdot p_1(v)-\frac{1}{2}p_2(v)\\
    &\ge \frac{9}{2}-\frac{1}{2}\cdot 2-1-\frac{1}{2}(4-p_3(v))\\
    &\ge \frac{1}{2}+\frac{1}{2} p_3(v),
\end{align*}
which implies the claim.
\end{proof}
Note that
\begin{equation*}
    \frac{p_3(v)-1}{2p_3(v)}=\frac{1}{2}-\frac{1}{2p_3(v)}
\end{equation*}
is increasing as a function of $p_3(v)$, so we can say that $r(v)$ is at least $\frac{1}{4}$ or $\frac{1}{3}$ when $p_3(v)$ is at least $2$ or $3$, respectively.
\begin{claim}
\label{lem2}
Suppose $v_1=(a+1,b),v_2=(a-1,b)\in S$, then 
\begin{equation*}
    r(v_1)+r(v_2)\ge \frac{1}{2}.
\end{equation*}
\end{claim}
\begin{proof}[Proof of Claim \ref{lem2}]
Suppose $r(v_1)+r(v_2)<\frac{1}{2}$, then from Claim \ref{lemr1}, we know that $v_1, v_2\in S_1$ and 
\begin{equation*}
N(v_i)\cap S=\{m(v_i)\}\quad \mbox{for}\quad i\in\{1,2\}.    
\end{equation*}
It indicates that $m(v_1)$ is either $(a+2,b+1)$ or $(a+2,b-1)$, so
\begin{equation*}
u_1=(a,b-1),u_2=(a,b),u_3=(a,b+1)    
\end{equation*}
are in $P(v_1)\backslash S$. From Claim \ref{lemr1}, we also know $u_1,u_2,u_3$ are not in $I$. Since $v_1$ and $v_2$ are in the neighborhood of $u_j$ for each $j\in\{1,2,3\}$, to locate $u_1,u_2$ and $u_3$, at least two of them has another neighbor in $S$. We conclude that at least two of $u_1,u_2,u_3$ are in $T_3\backslash I$. That means $p_3(v_1)\ge 2$ and $r(v_1)\ge \frac{1}{4}$. Similarly, we have $r(v_2)\ge \frac{1}{4}$. Here is a contradiction, since
\begin{equation*}
    r(v_1)+r(v_2)\ge \frac{1}{4}+\frac{1}{4}=\frac{1}{2}.
\end{equation*}
\end{proof}

Remark that from the symmetry of the king grid, this claim also works for the case when $v_1=(a,b+1), v_2=(a,b-1)$.

Now, we are ready to investigate the charge $\ch_4$ on $T_3\backslash I$. Suppose $u_0=(a,b)\in T_3\backslash I$, then it has at least 3 neighbors in $S$, and for each $v\in S\cap N(u_0)$, we have $u_0\in P_3(v)$. We separate it into three cases to check if $\ch_4(u_0)\ge 1$. Because of the symmetry of the king grid, each $u_0\in T_3\backslash I$ is symmetric to one of these three cases.

\begin{case1}
\label{case1}
$N(u_0)\cap S$ is not an independent set in the king grid $G$.
\end{case1}
Suppose $v_1,v_2$ are adjacent and $v_1,v_2\in S\cap N(u_0)$, then since $u_0\notin I$, $v_1,v_2$ are not partner of each other. $v_1$ is in either $I_0(v_2)$ or $P_0(v_2)$. From Claim \ref{lemr1}, we know that $r(v_2)=1/2$. Similarly, we have $r(v_1)=1/2$. Therefore, we can say
\begin{equation*}
    \ch_4(u_0)\ge g_2(v_1,u_0)+g_2(v_2,u_0)=r(v_2)+r(v_1)=1.
\end{equation*}

\begin{case2}
\label{case2}
$N(u_0)\cap S$ is an independent set in the king grid $G$ and $v_1=(a+1,b+1), v_2=(a+1,b-1), v_3=(a-1,b+1)\in N(u_0)\cap S$. See Figure \ref{c1}.
\end{case2}
Remark that in Case 2, the vertex $(a-1,b-1)$ can either be in $S$ or not.
\begin{figure}[h!]
\centering
\includegraphics[width=0.4\linewidth]{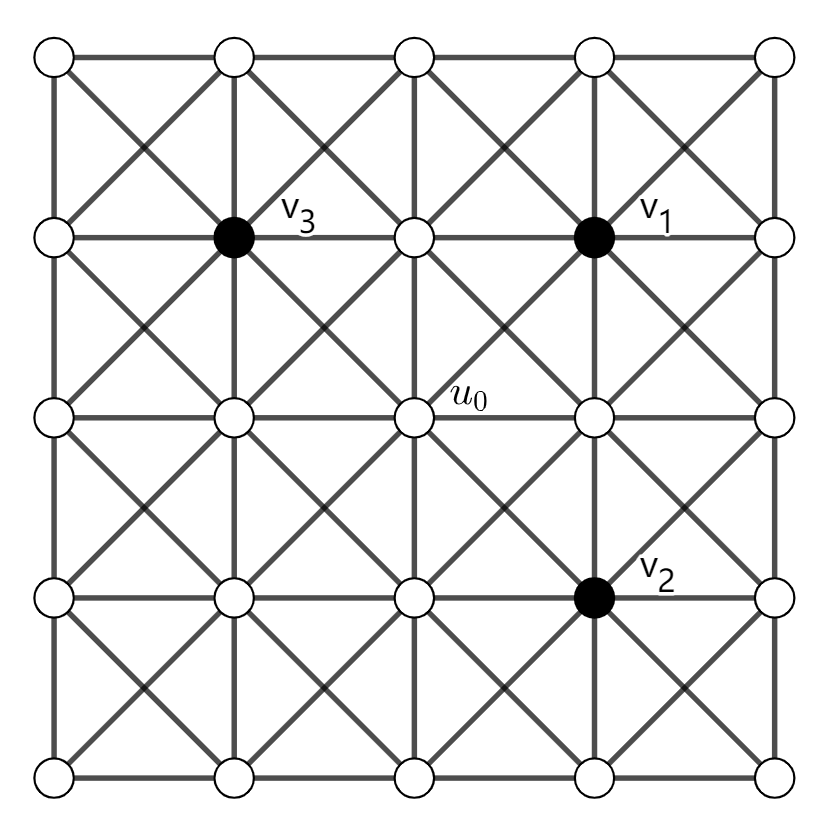}
\caption{Case 2 of the vertex $u_0$ in $T_3\backslash I$}
\label{c1}
\end{figure}
\begin{claim}
\label{clcase2}
In Case 2, we have $\ch_4(u_0)\ge 1$.
\end{claim}
\begin{proof}[Proof of Claim \ref{clcase2}]
Suppose $\ch_4(u_0)<1$, then
\begin{equation*}
    r(v_1)+r(v_2)+r(v_3)\le ch_4(u_0)<1.
\end{equation*}
From Claim \ref{lem2}, we have $r(v_1)+r(v_2)\ge \frac{1}{2}$ and $r(v_1)+r(v_3)\ge \frac{1}{2}$. Then, we know $r(v_2)<\frac{1}{2}$ and $r(v_3)<\frac{1}{2}$, which means they violate the conditions for Claim \ref{lemr1}. This makes sure
\begin{equation*}
    m(v_1)\notin\{(a,b+2),(a+2,b)\},
\end{equation*}
which means either $m(v_1)\in \{(a+1,b+2),(a+2,b+1)\}$ or $m(v_1)=(a+2,b+2)$. The former one gives $v_1\in S_2$ and the latter one gives $p_1(v_1)\le|T_1\cap P(v_1)|=0$. The conclusion is $r(v_1)=\frac{1}{2}$.

Let's move back to $v_2$. Since $r(v_2)<\frac{1}{2}$, we know $p_0(v_2)=0, p_1(v_2)=1$ and $v_2\in S_1$. Either $m(v_2)\in\{(a+2,b-2),(a,b-2)\}$ or $m(v_2)=(a+2,b)$. If $m(v_2)\in\{(a+2,b-2),(a,b-2)\}$, then $(a+2,b)$ is not in $S$. Consider
\begin{equation*}
  U=\{(a,b),(a+1,b),(a+2,b)\}.  
\end{equation*}
Since $U$ is in the neighborhood of both $v_1$ and $v_2$, to locate these three vertices in $U$, at least two of them are in $T_3$. Therefore, we have $p_3(v_2)\ge 2$ in this case. For the other case $m(v_2)=(a+2,b)$, we can say
\begin{equation*}
    \{(a,b),(a,b-1),(a,b-2),(a+1,b-2),(a+2,b-2)\}=P_1(v_2)\cup P_2(v_2)\cup P_3(v_2).
\end{equation*}
We claim that we also have $p_3(v_2)\ge 2$ in this case. If not, then $P_3(v_2)=\{(a,b)\}$ and
\begin{equation*}
    \{(a,b-1),(a,b-2),(a+1,b-2),(a+2,b-2)\}=P_1(v_2)\cup P_2(v_2).
\end{equation*}

When $P_1(v_2)=\{(a,b-1)\}$, $(a,b-2),(a+1,b-2)$ and $(a+2,b-2)$ are all in $P_2(v_2)$. That means there are three distinct vertices $v_4,v_5,v_6$ in $S$ such that 
\begin{align*}
    &\{v_4,v_2\}=N((a,b-2))\cap S,\\
    &\{v_5,v_2\}=N((a+1,b-2))\cap S,\\
    &\{v_6,v_2\}=N((a+2,b-2))\cap S.
\end{align*}
However, this is impossible since the first and the last one cover the middle one for those locations that can be in $S$.

When $P_1(v_2)\neq \{(a,b-1)\}$, we have $(a,b-1)\in P_2(v_2)$. The other neighbor of $(a,b-1)$ in $S$ is either $(a-1,b-1)$ or $(a-1,b-2)$. In either way, it violates $(a,b-2)\in P_2(v_2)\cup P_1(v_2)$, since
\begin{equation*}
     N((a,b-2))\cap S\neq N((a,b-1))\cap S.
\end{equation*}

The claim $p_3(v_2)\ge 2$ has been proved by contradiction. Then, in both cases, we have shown $p_3(v_2)\ge 2$, which means $r(v_2)\ge \frac{1}{4}$ by Claim \ref{lemr2}.

Similarly, we have the same argument that says $r(v_3)\ge \frac{1}{4}$. We conclude that
\begin{equation*}
    \ch_4(u_0)\ge r(v_1)+r(v_2)+r(v_3)\ge \frac{1}{2}+\frac{1}{4}+\frac{1}{4}=1.
\end{equation*}
\end{proof}

\begin{case3}
\label{case3}
$N(u_0)\cap S$ is an independent set in the king grid $G$ and $v_1=(a,b+1), v_2=(a+1,b-1), v_3=(a-1,b-1)\in N(u_0)\cap S$. See Figure \ref{c3}.
\end{case3}

\begin{figure}[h!]
\centering
\includegraphics[width=0.4\linewidth]{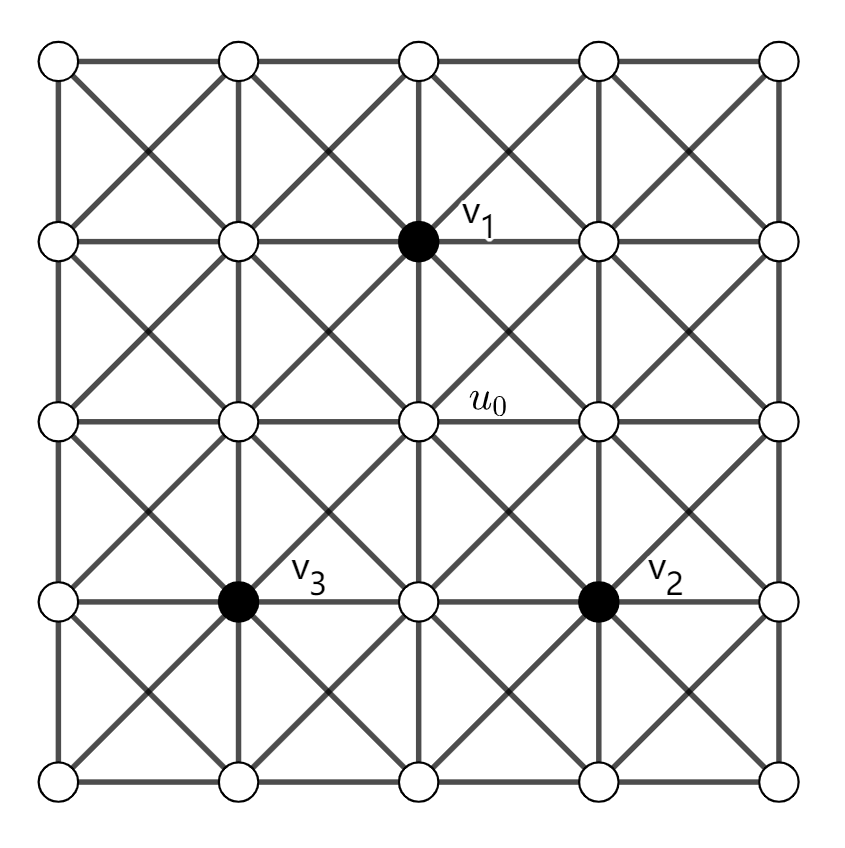}
\caption{Case 3 of the vertex $u_0$ in $T_3\backslash I$}
\label{c3}
\end{figure}

Firstly, It is easy to find that $\{v_1,v_2,v_3\}=N(u_0)\cap S$. This case is rather complicated. For those $u_0$ whose charge $\ch_4(u_0)$ is less than 1, we call them \textit{deficient vertices} and we need to find a \textit{rich friend} for each of them and give the discharge rule $g_3$ between the deficient vertices and their rich friends. Now, we assume $u_0$ is deficient, then
\begin{equation*}
\ch_4(u_0)=r(v_1)+r(v_2)+r(v_3)<1.
\end{equation*}
From Claim \ref{lem2}, we have $r(v_2)+r(v_3)\ge \frac{1}{2}$, then $r(v_1)<\frac{1}{2}$. From Claim \ref{lemr1}, we know that 
\begin{equation*}
    p_0(v_1)=0, m(v_1)\in \{(a-1,b+2),(a+1,b+2)\}.
\end{equation*}
Since $(a+1,b)\notin I$, we know
\begin{equation*}
    m(v_2)\in \{(a,b-2),(a+1,b-2),(a+2,b-2)\}.
\end{equation*}
Similarly,
\begin{equation*}
    m(v_3)\in \{(a,b-2),(a-1,b-2),(a-2,b-2)\}.
\end{equation*}
There are 9 cases for the positions of $m(v_2)$ and $m(v_3)$, but we can kill one case since $m(v_2)\neq m(v_3)$ and kill three other cases by symmetry, so we only need to consider the following five cases.
\begin{case31}
 $m(v_2)=(a+1,b-2), m(v_3)=(a-1,b-2)$.
\end{case31}
In this case, $v_2,v_3\in S_2$. From Claim \ref{lemr1}, we have $r(v_2)+r(v_3)\ge \frac{1}{2}+\frac{1}{2}=1$, which means $u_0$ is not deficient. Contradiction.
\begin{case32}
$m(v_2)=(a+1,b-2),m(v_3)=(a-2,b-2)$.
\end{case32}
In this case, $v_2\in S_2$ and $(a,b-1)\in I(v_2)\cap P(v_3)$. From Claim \ref{lemr1}, we have $r(v_2)+r(v_3)\ge \frac{1}{2}+\frac{1}{2}=1$, which means $u_0$ is not deficient. Contradiction.
\begin{case33}
$m(v_2)=(a,b-2),m(v_3)=(a-1,b-2)$.
\end{case33}
In this case, $u_0$ is possibly deficient. We assign the rich friend $(a,b-1)$ to it. Let
\begin{equation}
\label{g31}
    g_3((a,b-1),u_0)=1,
\end{equation}
then since $(a,b-1)\in I(v_2)\cap I(v_3)$, $\ch_4((a,b-1))\ge 2$ makes sure that $(a,b-1)$ and $u_0$ both have charge $\ch_5$ at least 1. It is trivial to check that one vertex cannot be a rich friend of two different $u_0$ in this case by investigating the neighborhood of the rich friend.
\begin{case34}
$m(v_2)=(a,b-2),m(v_3)=(a-2,b-2)$.
\end{case34}
In this case, $u_0$ is possibly deficient. When $u_0$ is deficient, we assign the rich friend $v_3$ to it. Let
\begin{equation}
\label{g32}
    g_3(v_3,u_0)=\frac{1}{2}.
\end{equation}
We can see that $(a,b-1)\in I\cap P(v_3)$ and $(a,b-2)\in S\cap P(v_3)$. We have $p_0(v_3)\ge 2$, which means
\begin{equation}
\label{c34v}
\begin{aligned}
    \ch_4(v_3)&\ge \frac{9}{2}-2\cdot \frac{1}{2}-p_1(v_3)-\frac{1}{2}p_3(v_3)-\frac{1}{2}p_3(v_3)\\
    &\ge 3-\frac{1}{2}(5-p_0(v_3))\\
    &=\frac{3}{2}.
\end{aligned}
\end{equation}
The charge $\ch_5$ for this case will be discussed after we define $g_3$ completely. For the discussion later, we would like to investigate the rich friend in this case now. Consider those four $\sqrt{2}$-neighbors of the rich friend $v_3$. One is a deficient vertex in Case 3. The opposite vertex of it is the partner $m(v_3)$ of the rich friend. A third one is a vertex in $S$.
\begin{case35}
$m(v_2)=(a+2,b-2),m(v_3)=(a-2,b-2)$.
\end{case35}
In this case, $(a,b-2)\in  P(v_2)\cap P(v_3)$. If $(a,b-2)\in I\cup S$, then from Claim \ref{lemr1}, we have $r(v_2)+r(v_3)=1$, which contradicts that $u_0$ is deficient. Also, since $(a,b-1)$ and $(a,b-2)$ are not in $S$ and
\begin{equation*}
    N((a,b-1))\cap S=\{v_2,v_3\}\subset N(a,b-2)\cap S,
\end{equation*}
to locate them, there exists
\begin{equation}
\label{c35c}
v_4\in \{(a-1,b-3),(a,b-3),(a+1,b-3)\}\cap S.    
\end{equation}
$u_0$ is also possibly deficient here. When $u_0$ is deficient, we assign the rich friend $v_4$ to it. Let
\begin{equation}
\label{g33}
    g_3(v_4,u_0)=\frac{1}{2}.
\end{equation}

According to \eqref{c35c}, we need to consider the following two cases by symmetry.
\begin{case351}
$v_4=(a-1,b-3)$.
\end{case351}
One useful fact is that 
\begin{equation}
\label{3511}
    m(v_4)\in \{(a-2,b-3),(a-2,b-4),(a-1,b-4),(a,b-4)\},
\end{equation}
since $(a,b-2)\notin I$. Consider those four $\sqrt{2}$-neighbor of the rich friend $v_4$. One vertex $s_1$ is in $T_3\backslash(I\cup S)$ in Case 2. Another non-opposite vertex $s_2$ is in $S$ and is not the partner $m(v_4)$,

Another useful fact is that 
\begin{equation}
\label{3512}
I(s_2)\cap N(v_4)= N(s_2)\cap N(s_1)\cap N(v_4).
\end{equation}
\begin{case352}
$v_4=(a,b-3)$.
\end{case352}
One useful fact is that 
\begin{equation}
\label{3521}
    m(v_4)\in \{(a-1,b-4),(a,b-4),(a+1,b-4)\}.
\end{equation}
since $(a,b-2)\notin I$. Consider those four $\sqrt{2}$-neighbors of the rich friend $v_4$. Two of them are vertices in $T_3\cap I$. They are not in the opposite position of each other.

Now, we have finished the definition of $g_3$, which is given by \eqref{g31}, \eqref{g32} and \eqref{g33}. If $\ch_5$ is at least $1$ everywhere in the kind grid. Under the definition of density on an infinite graph, it is not hard to find that the average charge for $\ch_2$ is the left side of \eqref{eqthm2} and the average charge for $\ch_5$ is at least 1. The average charge is fixed under the discharge rule, so Theorem \ref{eqthm2} holds. The remaining part of this section is to prove the following claim.
\begin{claim}
\label{lemmain}
Under the discharge rule $g_3$, the charge $\ch_5$ is at least one everywhere.
\end{claim}
\begin{proof}[Proof of Claim \ref{lemmain}]
We have shown that $\ch_4$ is at least 1 for all vertices except for the deficient vertices in Case 3.3, Case 3.4, and Case 3.5. It suffices to prove that $\ch_5$ is at least one for those related vertices, which are all deficient vertices and their rich friends. In Case 3, by Claim \ref{lem2} we have 
\begin{equation*}
    \ch_4(u_0)\ge r(v_2)+r(v_3)\ge \frac{1}{2}.
\end{equation*}
From \eqref{g31}, \eqref{g32} and \eqref{g33}, we know
\begin{equation*}
    \ch_5(u_0)\ge ch_4(u_0)+\frac{1}{2}\ge 1,
\end{equation*}
which means all deficient vertices have charge $\ch_5$ at least one. 

The rich friend in Case 3.3 is not in $S$, so it cannot be a rich friend in any other case at the same time. We have shown its charge is at least one. Consider other rich friends. For convenience, we call the rich friends in Case 3.4 type A rich friend, call the rich friends in Case 3.5.1 type B rich friend, and call the rich friends in Case 3.5.2 type C rich friend. We have given the observations on the four $\sqrt{2}$-neighbors for these three types of rich friends. 

Suppose $v_0$ is a rich friend of type A, then the observations about those three $\sqrt{2}$-neighbors in Case 3.3 prevent $v_0$ from being a rich friend of type B or C and prevent $v_0$ from being a rich friend of type A for two different deficient vertices. It means type A rich friends have charge $\ch_5$ at least one because of \eqref{c34v} and \eqref{g32}.

Type B rich friend has one $\sqrt{2}$-neighbor in $S$ and one $\sqrt{2}$-neighbor in $T_3\backslash (S\cup I)$. They are not in the opposite position of each other and the deficient vertex is determined by these two $\sqrt{2}$-neighbors. Type C rich friend has two $\sqrt{2}$-neighbors in $I\backslash S$. They are not in the opposite position of each other and the deficient vertex is determined by these two $\sqrt{2}$-neighbors. Suppose $v_0$ is a rich friend of type B, then it is in one of the following three cases.  

If $v_0$ is a rich friend of type B and also a rich friend of type C, then from the facts \eqref{3511}\eqref{3521}, we have no valid position for the partner $m(v_0)$, which gives a contradiction.

If $v_0$ is a rich friend of type B for deficient vertex $u_1^{\prime}$ and $v_0$ is also a rich friend of type B for another deficient vertex $u_2^{\prime}$, then from the facts \eqref{3511}\eqref{3512}, it is not hard to find that $v_0\in S_2$ and $v_0$ is not a rich friend for another deficient vertex different from $u_1^{\prime},u_2^{\prime}$. Also, 
\begin{equation*}
    |S\cap N(v_0)|\ge 3
\end{equation*}
indicates $i_0(v_0)+p_0(v_0)\ge 2$. Then,
\begin{equation*}
\begin{aligned}
    \ch_5(v_0)&\ge -g(v_0,u_1^{\prime})-g(v_0,u_2^{\prime})+\ch_4(v_0)\\
    &\ge -1+5-\frac{1}{2}|I(v_0)\backslash S|-p_1(v_0)-\frac{1}{2}p_2(v_0)-\frac{1}{2}p_3(v_0)\\
    &\ge 4-\frac{1}{2}(4-i_0(v_0))-\frac{1}{2}p_1(v_0)-\frac{1}{2}(3-p_0(v_0))\\
    &\ge 0+\frac{1}{2}(i_0(v_0)+p_0(v_0))\\
    &\ge 1.
\end{aligned}
\end{equation*}

If $v_0$ is only a rich friend of type B for deficient vertex $u_1^{\prime}$ and not a rich friend for another deficient vertex different from $u_1^{\prime}$, then the non-partner $S$ neighbor indicates $i_0(v_0)+p_0(v_0)\ge 1$. Then,
\begin{equation*}
\begin{aligned}
    \ch_5(v_0)&\ge -g(v_0,u_1^{\prime}) +\ch_4(v_0)\\
    &\ge -\frac{1}{2}+5-\frac{1}{2}|I(v_0)\backslash S|-p_1(v_0)-\frac{1}{2}p_2(v_0)-\frac{1}{2}p_3(v_0)\\
    &\ge -\frac{1}{2}+5-\frac{1}{2}(4-i_0(v_0))-\frac{1}{2}p_1(v_0)-\frac{1}{2}(3-p_0(v_0))\\
    &\ge \frac{1}{2}+\frac{1}{2}(i_0(v_0)+p_0(v_0))\\
    &\ge 1.
\end{aligned}
\end{equation*}

Now, we have shown any rich friend of type B has charge $\ch_5$ at least one. Suppose $v_0$ is a rich friend of type C, then we have the following two cases remaining.

If $v_0$ is a rich friend of type C for deficient vertex $u_1^{\prime}$ and $v_0$ is also a rich friend of type C for another deficient vertex $u_2^{\prime}$, then from the facts\eqref{3521}, we know three of those four $\sqrt{2}$-neighbors of $v_0$ are in $I\backslash S$ and the fourth one is the partner $m(v_0)$. Then, $p_0(v_0)\ge 3$ and
\begin{equation*}
\begin{aligned}
    \ch_5(v_0)&\ge -g(v_0,u_1^{\prime})-g(v_0,u_2^{\prime}) +\ch_4(v_0)\\
    &\ge -1+\frac{9}{2}-\frac{1}{2}|I(v_0)\backslash S|-p_1(v_0)-\frac{1}{2}p_2(v_0)-\frac{1}{2}p_3(v_0)\\
    &\ge -1+\frac{9}{2}-\frac{1}{2}*2-\frac{1}{2}p_1(v_0)-\frac{1}{2}(5-p_0(v_0))\\
    &\ge -\frac{1}{2}+\frac{1}{2}(p_0(v_0))\\
    &\ge 1.
\end{aligned}
\end{equation*}

If $v_0$ is only a rich friend of type C for deficient vertex $u_1^{\prime}$ and not a rich friend for another deficient vertex different from $u_1^{\prime}$, then the two $\sqrt{2}$-neighbors in $I\backslash S$ and the fact\eqref{3521} indicates $p_0(v_0)\ge 2$. Then,
\begin{equation*}
\begin{aligned}
    \ch_5(v_0)&\ge -g(v_0,u_1^{\prime}) +\ch_4(v_0)\\
    &\ge -\frac{1}{2}+\frac{9}{2}-\frac{1}{2}|I(v_0)\backslash S|-p_1(v_0)-\frac{1}{2}p_2(v_0)-\frac{1}{2}p_3(v_0)\\
    &\ge -\frac{1}{2}+\frac{9}{2}-\frac{1}{2}*2-\frac{1}{2}p_1(v_0)-\frac{1}{2}(5-p_0(v_0))\\
    &\ge \frac{1}{2}(p_0(v_0))\\
    &\ge 1.
\end{aligned}
\end{equation*}
\end{proof}
We have shown the charge $\ch_5$ is at least one everywhere in $G$, which finishes the proof of Theorem \ref{thm2}.
\end{proof}

%%%%%%%%%%%%%%%
%
% SECTION 5
%
%%%%%%%%%%%%%%%
\section{More LPDS patterns}
To give a lower bound of the density of an optimal LPDS on the king grid, Hussain, Niepel, and Kinawi gave the tiling as shown in Figure \ref{l1}. To be specific,
\begin{equation}
    L_1=\{a+k_1(2,1)+k_2(-3,3)\in \mathbb{Z}^2|a=(0,0) \mbox{ or } (-1,1),(k_1,k_2)\in \mathbb{Z}^2\}.
\end{equation}

\begin{figure}[h!]
\centering
\includegraphics[width=0.75\linewidth]{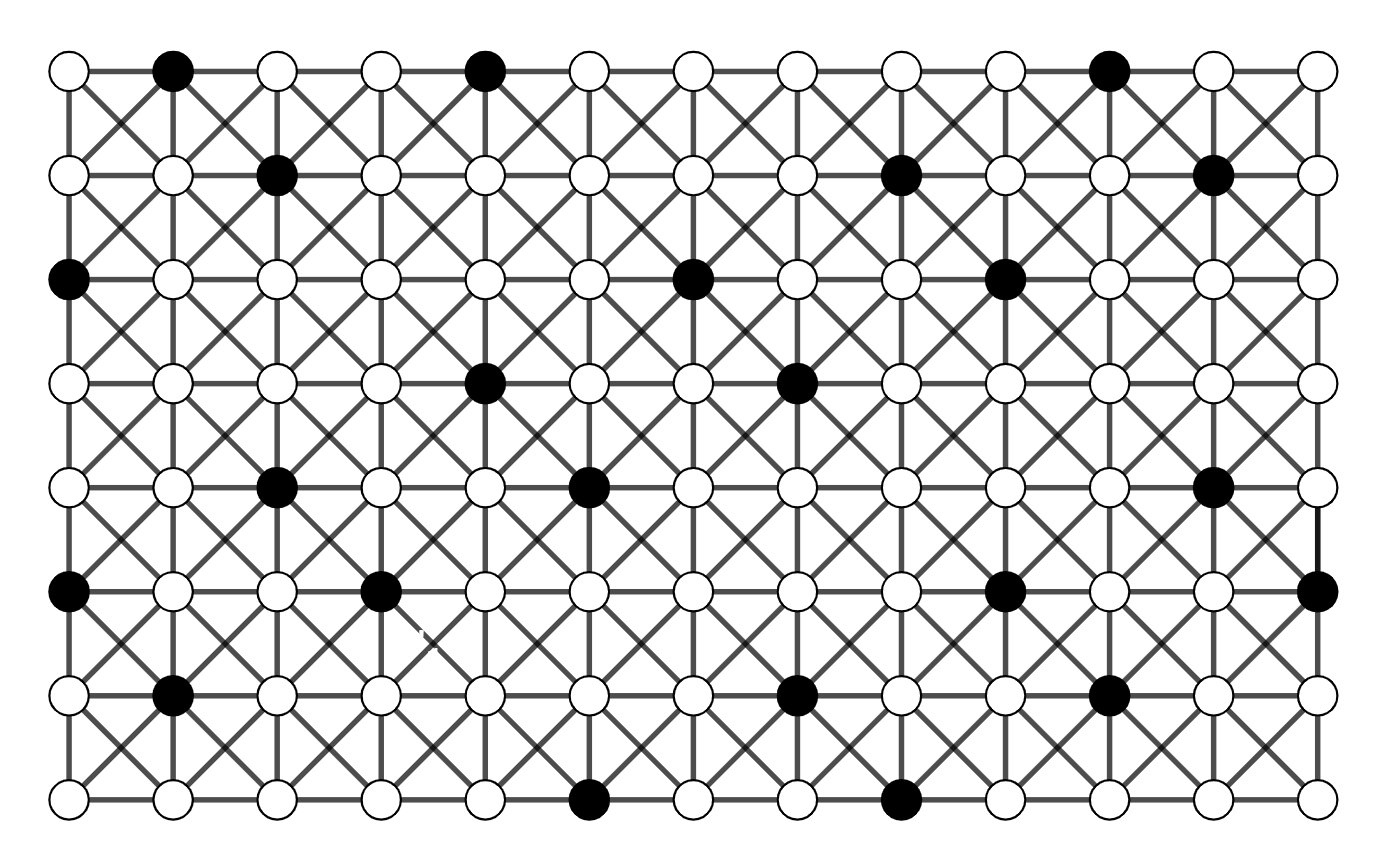}
\caption{ the LPDS $L_1$ in the king grid}
\label{l1}
\end{figure}

It is not hard to find that there are only finite many different LPDS obtained from the symmetric transformation of $L_1$.

Let $L_0=\{(0,0),(0,3),(2,2),(3,1),(4,3),(5,0),(7,1),(7,2)\}$, then a different LPDS is given by
\begin{equation}
    L_2=\{a+k_1(9,0)+k_2(0,4)\in \mathbb{Z}^2|a\in L_0,(k_1,k_2)\in \mathbb{Z}^2\},
\end{equation}
as shown in Figure \ref{l2}.

\begin{figure}[h!]
\centering
\includegraphics[width=0.75\linewidth]{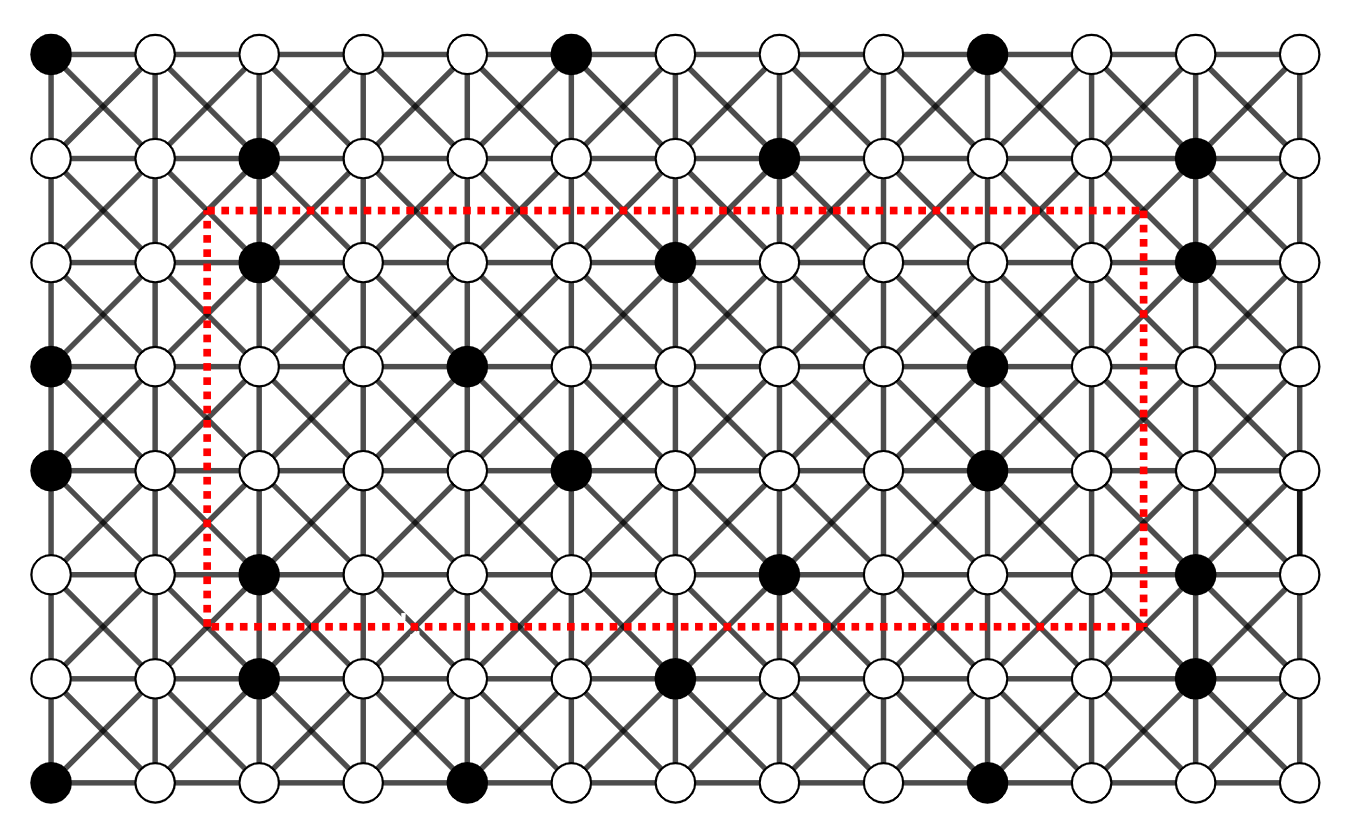}
\caption{ the LPDS $L_2$ in the king grid}
\label{l2}
\end{figure}

It is trivial to check that $L_2$ is a LPDS of the king grid.
\begin{proof}[Proof of Theorem \ref{thm4}]
For the uncountable set $\{X|X\subset \mathbb{Z}\}$, we construct a LPDS 
\begin{equation*}
    L_X=\{a+k_1(9,0)+k_2(0,4)+1_X(k_1)(0,1)|a\in L_0,(k_1,k_2)\in \mathbb{Z}^2\},
\end{equation*}
where $1_X(k_1)=1$ when $k_1\in X$ and $1_X(k_1)=0$ when $k_1\notin X$. Remark that $L_2$ is just $L_{\emptyset}$.

It is not hard to check that $L_X$ is a LPDS with density $\frac{2}{9}$ in the king grid. Apparently, $X\rightarrow L_X$ is injective.
\end{proof}
Remark that the cardinality of the automorphism group of the king grid is countable, so we can also say there are uncountable many non-isomorphic LPDS in the king grid.

It is interesting to see that $T_3\backslash I$ is empty for all the examples in this section, but, as you can see, $T_3\backslash I$ is the main difficulty in the proof of Theorem \ref{thm2} in the last section. This fact may be a hint for future research.

\section*{Acknowledgements}
The author would like to thank Mei Lu for introducing related topics and providing useful comments.

%%%%%%%%%%
%
% REFERENCES
%
%%%%%%%%%%

\end{document}